\documentclass[12pt,reqno]{amsart}
\usepackage{amsmath,amssymb,amsthm,mathtools,enumerate}
\calclayout
\theoremstyle{plain}
\newtheorem{thm}{Theorem}[section]
\newtheorem{theorem}[thm]{Theorem}
\newtheorem{proposition}[thm]{Proposition}
\newtheorem{lemma}[thm]{Lemma}
\newtheorem{corollary}[thm]{Corollary}
\theoremstyle{definition}
\newtheorem{definition}[thm]{Definition}
\newtheorem{remark}[thm]{Remark}
\newcommand{\be}{\boldsymbol{e}}
\newcommand{\bk}{\boldsymbol{k}}
\newcommand{\bl}{\boldsymbol{l}}

\newcommand{\bQ}{\mathbb{Q}}
\newcommand{\bR}{\mathbb{R}}
\newcommand{\jump}[1]{\ensuremath{[\![#1]\!]}}
\title{The connector for Double Ohno relation}
\author{Minoru Hirose, Nobuo Sato and Shin-ichiro Seki}
\address{Faculty of mathematics, Kyushu university 744, Motooka, Nishi-ku, Fukuoka, 819-0395, Japan}
\email{m-hirose@math.kyushu-u.ac.jp}
\address{Faculty of mathematics, Kyushu university 744, Motooka, Nishi-ku, Fukuoka, 819-0395, Japan}
\email{n-sato@math.kyushu-u.ac.jp}
\address{Mathematical Institute, Tohoku University, 6-3, Aoba, Aramaki, Aoba-Ku,
Sendai, 980-8578, Japan}
\email{shinichiro.seki.b3@tohoku.ac.jp}
\thanks{This work was supported by JSPS KAKENHI Grant Numbers JP18J00982, JP18K13392, JP19J00835, and JP18J00151.}
\subjclass[2010]{11M32, 11B65.}
\keywords{Multiple zeta values, connector, connected sum, Ohno relation, Double Ohno relation.}
\begin{document}
\begin{abstract}
In this paper, we introduce a new connector which generalizes the connector found by the third author and Yamamoto.
The new connector gives a direct proof of the double Ohno relation recently proved by the first author, the second author, Murahara, and Onozuka.
Furthermore, we obtain a simultaneous generalization of the ($q$-)Ohno relation and the ($q$-)double Ohno relation.
\end{abstract}
\maketitle
\section{Introduction}
Ohno \cite{O} proved the \emph{Ohno relation} which is a family of $\bQ$-linear relations among multiple zeta values stated as
\[
\sum_{|\be|=e}\zeta(\bk+\be)=\sum_{|\be'|=e}\zeta(\bk^{\dagger}+\be')
\]
for each admissible index $\bk$ and each non-negative integer $e$.
Here, $\bk^{\dagger}$ is the dual index of $\bk$ and $\sum_{|\be|=e}$ (resp.~$\sum_{|\be'|=e}$) means that $\be$ (resp.~$\be'$) runs over $r$-tuples (resp.~$r'$-tuples) of non-negative integers whose sum of all components equals $e$ when $\bk$ (resp.~$\bk^{\dagger}$) is an $r$-tuple (resp.~an $r'$-tuple).

In \cite{HMOS}, the authors proved the \emph{double Ohno relation} which is a double version of the Ohno relation, i.e.,
\[
\sum_{\substack{|\be_1|=e_1 \\ |\be_2|=e_2}}\zeta(\bk+\be_1+\be_2)=\sum_{\substack{|\be_1'|=e_1 \\ |\be_2'|=e_2}}\zeta(\bk^{\dagger}+\be_1'+\be_2')
\]
for certain specific $\bk$'s called BBBL-type indices and any non-negative integers $e_1,e_2$ by assuming the Ohno relation.
Although the double Ohno relation gives the original single version by setting either $e_1$ or $e_2$ to be zero, it does not recover the Ohno relation in full generality since its index $\bk$ is restricted.

The purpose of this article is to prove an identity which generalizes both the single and the double Ohno relation.
Our proof of the identity makes use of the technique called the \emph{connected sum} method, which was originally invented and developed by the third author and Yamamoto in their paper \cite{SY}.
They exhibit an ingenious simple proof of the Ohno relation (as well as its $q$-analog).
Our proof enhances theirs and at the same time provides a direct proof of the double Ohno relation that does not depend on the Ohno relation.
The identity also gives an explanation why the double Ohno relation holds only for restricted indices.

Our method is also applicable to a certain type of $q$-multiple zeta values and the case of usual multiple zeta values is recovered by taking $q\to1$.
Therefore, we discuss everything under the $q$-setting, and sometimes omit $q$ from our notation.

Let $q$ be a real number satisfying $0<q<1$.
For a positive integer $m$, the \emph{$q$-integer} $[m]$ is defined to be $\frac{1-q^m}{1-q}$.
In this paper, we call a non-empty tuple of positive integers $\bk=(k_1,\dots,k_r)$ with $k_r\geq 2$ an \emph{admissible index} and we define the $q$-multiple zeta value $\zeta_q(\bk)$ for such an index by the following convergent series:
\[
\zeta_q(\bk)\coloneqq\sum_{0<m_1<\cdots<m_r}\frac{q^{(k_1-1)m_1+\cdots+(k_r-1)m_r}}{[m_1]^{k_1}\cdots [m_r]^{k_r}}.
\]
In order to state our identity, it is convenient to introduce the generating function of double Ohno sums.
Here, the \emph{double Ohno sum} is defined by
\[
O_{e_1,e_2}(\bk)\coloneqq\sum_{\substack{|\be_1|=e_1 \\ |\be_2|=e_2}}\zeta_q(\bk+\be_1+\be_2)
\]
for an admissible index $\bk$ and non-negative integers $e_1,e_2$.
Throughout this paper, we fix non-negative real numbers $\xi,\eta$ which are sufficiently small depending only on $q$. Then, the desired generating function is
\[
O(\bk)=O_q(\bk;\xi,\eta)\coloneqq\sum_{0<m_1<\cdots<m_r}\prod_{i=1}^r\frac{q^{(k_i-1)m_i}}{([m_i]-q^{m_i}\xi)([m_i]-q^{m_i}\eta)[m_i]^{k_i-2}}.
\]
Since $\frac{1}{[m]-q^m\xi}\leq\frac{1}{1-\xi}\cdot\frac{1}{[m]}$ and $\frac{1}{[m]-q^m\eta}\leq\frac{1}{1-\eta}\cdot\frac{1}{[m]}$ for each positive integer $m$, we see that
\[
O(\bk)\leq\left(\frac{1}{(1-\xi)(1-\eta)}\right)^r\zeta_q(\bk)
\]
and $O(\bk)$ is convergent.
By using the geometric series expansion, one may easily check that
\[
O(\bk)=\sum_{e_1,e_2\geq0}O_{e_1,e_2}(\bk)\xi^{e_1}\eta^{e_2}
\]
(cf.~\cite[Proof of Theorem~1.2]{SY}).
Here, we recall the definition of the dual index.
When a given admissible index $\bk$ is uniquely expressed as
\[
\bk=(\{1\}^{a_1-1},b_1+1,\dots,\{1\}^{a_s-1},b_s+1)
\]
by using positive integers $a_1,\dots,a_s,b_1,\dots,b_s$ ($\{1\}^a$ means $\overbrace{1,\dots,1}^a$), then the dual $\bk^{\dagger}$ is obtained as
\[
\bk^{\dagger}=(\{1\}^{b_s-1},a_s+1,\dots,\{1\}^{b_1-1},a_1+1).
\]
Furthermore, we call an admissible index $\bk$ \emph{BBBL-type} after \cite[(16)]{BBBL} when
\[
\bk=(\{2\}^{n_0},1,\{2\}^{n_1},3,\dots,\{2\}^{n_{2d-2}},1,\{2\}^{n_{2d-1}},3,\{2\}^{n_{2d}})
\]
for some non-negative integers $d,n_0,\dots,n_{2d}$.
Note that if $\bk$ is of BBBL-type, then $\bk^{\dagger}$ is also of BBBL-type.
\begin{theorem}[Double Ohno relation \cite{HMOS}]\label{thm:genDOR}
Let $\bk$ be an index of BBBL-type. Then, we have
\[
O(\bk)=O(\bk^{\dagger}).
\]
\end{theorem}
Next, we state our new identity.
For the convenience of the description, we identify an admissible index $(k_1,\dots,k_r)$ and a word $yx^{k_1-1}\cdots yx^{k_r-1}\in y\bQ\langle x,y\rangle x$, where $\bQ\langle x,y\rangle$ is the non-commutative polynomial ring in two variables $x,y$.
By this identification and the $\bQ$-linearity, we understand $O( \ )$ as a map from $y\bQ\langle x,y\rangle x$ to $\bR$.
Further, for an indeterminate $\lambda$, we set
\begin{align*}
\tau_{\lambda}(x)&\coloneqq(1+yx\lambda)^{-1}y=y(1+xy\lambda)^{-1},\\
\tau_{\lambda}(y)&\coloneqq x(1+yx\lambda)=(1+xy\lambda)x
\end{align*}
as elements of the formal power series ring $\bQ\langle x,y\rangle\jump{\lambda}$.
We define $A$ to be the subring of $\bQ\langle x,y\rangle\jump{\lambda}$ generated by $x,y,\lambda$ and $(1+yx\lambda)^{-1}$, and extend $\tau_{\lambda}$ to be the $\bQ$-anti-automorphism of $A$ by $\tau_{\lambda}(\lambda)=\lambda$.
Note that $\tau_{\lambda}((1+yx\lambda)^{-1})=(1+yx\lambda)^{-1}$.
Finally, for $w=\sum_{i=0}^{\infty}w_i\lambda^i\in A$, we extend $O( \ )$ as
\[
O(ywx)\coloneqq\sum_{i=0}^{\infty}O(yw_ix)\xi^i\eta^i.
\]
\begin{theorem}[Extended double Ohno relation]\label{thm:main}
For any $w\in A$, the series $O(ywx)$ is absolutely convergent and we have
\[
O(ywx)=O(y\tau_{\lambda}(w)x).
\]
\end{theorem}
Since $\tau_{\lambda}(xy)=xy$ and $\tau_{\lambda}(yx)=yx$ by definition, we see that the restriction of Theorem~\ref{thm:main} to $w\in\bQ\langle xy,yx\rangle$ gives the double Ohno relation.
Also, by setting $\xi=0$ or $\eta=0$, Theorem~\ref{thm:main} reduces to the single Ohno relation for $q$-multiple zeta values, which was originally proved by Bradley in \cite{B}.
Moreover, our new identity can be used to describe the difference $O_{e_1,e_2}(\bk)-O_{e_1,e_2}(\bk^{\dagger})$ for an arbitrarily admissible index $\bk$.
\section{Connector, connected sum, and transport relations}
For a non-negative integer $a$ and a real number $\delta$, $[a;\delta]$ denotes $\prod_{i=1}^a([i]-q^i\delta)$ ($[0;\delta]$ means $1$).
We set $[a]!\coloneqq[a;0]$.
\begin{definition}[Connector]
For non-negative integers $m$ and $n$, we define a connector $c(m,n)$ by
\[
c(m,n)\coloneqq\frac{q^{mn}[m;\xi][m;\eta][n;\xi][n;\eta]}{[m]![n]![m+n;\xi+\eta+(1-q)\xi\eta]}.
\]
\end{definition}
From now on, we consider more general indices including non-admissible indices and the empty index.
Namely, for a tuple of positive integers $\bk=(k_1,\dots,k_r)$, we admit the case $k_r=1$ or $\bk=\varnothing$ (the empty tuple).
When $\bk=\varnothing$, we consider that $r$ called `the depth of $\bk$' is zero.
\begin{definition}[Connected sum]
Let $\bk=(k_1, \dots, k_r)$ and $\bl=(l_1, \dots, l_s)$ be indices.
Then, we define a (possibly divergent) series of positive terms
$Z(\bk;\bl)=Z_q(\bk;\bl;\xi,\eta)$ by
\begin{align*}
Z(\bk;\bl)
\coloneqq\sum_{\substack{0=m_0<m_1<\cdots<m_r \\ 0=n_0<n_1<\cdots<n_s}}
&\prod_{i=1}^r\frac{q^{(k_i-1)m_i}}{([m_i]-q^{m_i}\xi)([m_i]-q^{m_i}\eta)[m_i]^{k_i-2}}\\
\cdot c(m_r,n_s)\cdot&\prod_{j=1}^s\frac{q^{(l_j-1)n_j}}{([n_j]-q^{n_j}\xi)([n_j]-q^{n_j}\eta)[n_j]^{l_j-2}}.
\end{align*}
Here, an empty product is regarded as $1$.
\end{definition}
By definition, we see that there is a symmetry $Z(\bk;\bl)=Z(\bl;\bk)$.
If either $\xi=0$ or $\eta=0$, then $c(m,n)$ (resp.~$Z(\bk;\bl)$) is nothing but the connector (resp.~the connected sum) defined by the third author and Yamamoto in \cite{SY}.
Note that $Z(\bk;\varnothing)$ for an admissible index $\bk$ is not necessarily equal to $O(\bk)$ unlike the case of either $\xi=0$ or $\eta=0$.
This is the reason why we need Proposition~\ref{prop:initial} in the next section.
\begin{lemma}\label{lem:limit}
If $n$ is at least $1$, then
\[
\lim_{m\to\infty}c(m,n)=0.
\]
\end{lemma}
\begin{proof}
Note that for $0\leq\delta\leq 1$, we have $[i-1]\leq[i]-q^i\delta\leq[i]$ for $i\geq 2$, and thus $(1-q\delta)[a-1]!\leq[a;\delta]\leq[a]!$ for $a\geq 1$.
Therefore, by setting $\gamma\coloneqq\xi+\eta+(1-q)\xi\eta<1$,
\[
c(m,n)\leq \frac{q^{mn}[m]!}{(1-q\gamma)[m+n-1]!}\cdot\frac{[n;\xi][n;\eta]}{[n]!}\leq (\text{a const.~indep.~of $m$})\cdot q^{mn}\xrightarrow{m\to\infty}0.\qedhere
\]
\end{proof}
For positive integers $k_1,\dots, k_r$ and an index $\bk=(k_1,\dots,k_r)$, we use the following arrow notation:
\begin{align*}
\bk_{\to}&\coloneqq(k_1,\dots,k_r,1),\\
\bk_{\uparrow}&\coloneqq(k_1,\dots,k_{r-1},k_r+1),\\
\bk_{\to\uparrow}&\coloneqq(\bk_{\to})_{\uparrow}=(k_1,\dots,k_r,2),\\
\bk_{\uparrow\to}&\coloneqq(\bk_{\uparrow})_{\to}=(k_1,\dots,k_{r-1},k_r+1,1),\\
\bk_{\uparrow\to\uparrow}&\coloneqq((\bk_{\uparrow})_{\to})_{\uparrow},\quad \text{and so on}.
\end{align*}

We also use $\varnothing_{\to}=(1)$, but we do not define $\varnothing_{\uparrow}$.
\begin{theorem}[Transport relations for the extended double Ohno relation]\label{thm:transport}
Let $\bk$ and $\bl$ be $($possibly non-admissible or empty$)$ indices.
If $\bl$ is not empty, then we have
\begin{equation}\label{eq:TR1}
Z(\bk_{\to};\bl)=Z(\bk;\bl_{\uparrow})+\xi\eta\cdot Z(\bk_{\to\uparrow};\bl_{\uparrow}).
\end{equation}
If $\bk$ is not empty, then we have
\begin{equation}\label{eq:TR2}
Z(\bk_{\uparrow};\bl)=Z(\bk;\bl_{\to})-\xi\eta\cdot Z(\bk_{\uparrow};\bl_{\to\uparrow}).
\end{equation}
\end{theorem}
\begin{proof}
Let $m,n$ be integers satisfying $m\geq 0, n\geq 1$.
Since we can easily check
\[
[a]\{[a+n]-q^{a+n}(\xi+\eta+(1-q)\xi\eta)\}-q^n([a]-q^a\xi)([a]-q^a\eta)=[a][n]-q^{a+n}\xi\eta,
\]
we have
\[
c(a-1,n)-c(a,n)=\frac{[a][n]-q^{a+n}\xi\eta}{q^n([a]-q^a\xi)([a]-q^a\eta)}\cdot c(a,n)
\]
for any positive integer $a$.
By taking the telescoping series, thanks to Lemma~\ref{lem:limit}, we obtain
\begin{equation}\label{eq:key}
\begin{split}
&\sum_{a=m+1}^{\infty}\frac{[a]}{([a]-q^a\xi)([a]-q^a\eta)}\cdot c(a,n)\\
&=\frac{q^n}{[n]}\cdot c(m,n)+\xi\eta\cdot\frac{q^n}{[n]}\sum_{a=m+1}^{\infty}\frac{q^a}{([a]-q^a\xi)([a]-q^a\eta)}\cdot c(a,n).
\end{split}
\end{equation}
This proves the first desired equality \eqref{eq:TR1}.
The second equality \eqref{eq:TR2} is equivalent to the first by the symmetry of the connected sum.
\end{proof}
\begin{remark}
Two series in \eqref{eq:key} converge.
We see this fact from a calculation in the proof of Proposition~\ref{prop:initial} below.
\end{remark}
\begin{proposition}\label{porp:conv}
Let $\bk$ and $\bl$ be indices.
Assume that one of the following conditions are satisfied$:$
\begin{enumerate}[$(i)$]
\item\label{item:non-empty} Both $\bk$ and $\bl$ are non-empty.
\item\label{item:empty} One of $\bk$ and $\bl$ is empty and the other is empty or admissible.
\end{enumerate}
Then, the connected sum $Z(\bk;\bl)$ converges.
\end{proposition}
\begin{proof}
Let $Y(\bk;\bl)=Z_q(\bk;\bl;0,0)$.
By the inequalities already used in this paper, we can check that
\[
Z(\bk;\bl)\leq\frac{1}{(1-q)(1-q\gamma)}\cdot\left(\frac{1}{(1-\xi)(1-\eta)}\right)^{r+s}Y(\bk;\bl),
\]
where $\gamma=\xi+\eta+(1-q)\xi\eta$.
Hence, it is sufficient to show the convergence of $Y(\bk;\bl)$ under the condition \eqref{item:non-empty} or \eqref{item:empty}.
If the condition \eqref{item:non-empty} is satisfied, then
\[
Y(\bk;\bl)=\cdots=Y((1);\boldsymbol{h})=Y(\varnothing;\boldsymbol{h}_{\uparrow})
\]
holds for some non-empty index $\boldsymbol{h}$ by the transport relations of $Y( \ ; \ )$.
Thus, $Y(\varnothing;\boldsymbol{l})$ for an empty or admissible index $\boldsymbol{l}$ is the only case to consider.
By definition, $Y(\varnothing;\boldsymbol{l})=\zeta_q(\boldsymbol{l})$ holds and this is convergent.
Here, we set $\zeta_q(\varnothing)\coloneqq 1$.
\end{proof}
\begin{corollary}[Transport relations for the double Ohno relation]\label{cor:transport}
\begin{align}
Z(\bk_{\to\uparrow};\bl)&=Z(\bk;\bl_{\to\uparrow}),\label{eq:TR3}\\
Z(\bk_{\uparrow\to};\bl)&=Z(\bk;\bl_{\uparrow\to})\quad (\bk,\bl\neq\varnothing).
\end{align}
\end{corollary}
\begin{proof}
If we apply the transport relations in Theorem~\ref{thm:transport} twice, then the terms with $\xi\eta$ cancel each other out and vanish as follows!
\begin{align*}
Z(\bk_{\to\uparrow};\bl)&\stackrel{\eqref{eq:TR2}}{=}Z(\bk_{\to};\bl_{\to})-\xi\eta\cdot Z(\bk_{\to\uparrow};\bl_{\to\uparrow})\\
&\stackrel{\eqref{eq:TR1}}{=}Z(\bk;\bl_{\to\uparrow})+\xi\eta\cdot Z(\bk_{\to\uparrow};\bl_{\to\uparrow})-\xi\eta\cdot Z(\bk_{\to\uparrow};\bl_{\to\uparrow})\\
&=Z(\bk;\bl_{\to\uparrow}),
\end{align*}
\begin{align*}
Z(\bk_{\uparrow\to};\bl)&\stackrel{\eqref{eq:TR1}}{=}Z(\bk_{\uparrow};\bl_{\uparrow})+\xi\eta\cdot Z(\bk_{\uparrow\to\uparrow};\bl_{\uparrow})\\
&\stackrel{\eqref{eq:TR2}}{=}Z(\bk;\bl_{\uparrow\to})-\xi\eta\cdot Z(\bk_{\uparrow};\bl_{\uparrow\to\uparrow})+\xi\eta\cdot Z(\bk_{\uparrow\to\uparrow};\bl_{\uparrow})\\
&\stackrel{\eqref{eq:TR3}}=Z(\bk;\bl_{\uparrow\to}).\qedhere
\end{align*}
\end{proof}
\section{Consequences of transporting indices}
For a real number $\delta$, we denote by $\delta'$ the number $1+(1-q)\delta$.
Let $\xi, \eta$ be as in the previous section and let $\gamma$ be $\xi+\eta+(1-q)\xi\eta$.
Note that $\gamma'=\xi'\eta'$.
From now on, we use usual $q$-Pochhammer symbols $(a;q)_m$ and $(a;q)_{\infty}=\lim\limits_{m\to\infty}(a;q)_m$.
Since $(1-q)([i]-q^i\delta)=1-q^i\delta'$, $(1-q)^m[m;\delta]=(q\delta';q)_m$ holds.
\begin{proposition}[Initial value]\label{prop:initial}
For a non-empty index $\bk$, the following identity holds:
\[
Z(\bk;(1))=Z((1);\bk)=\frac{(q\xi';q)_{\infty}(q\eta';q)_{\infty}}{(q;q)_{\infty}(q\xi'\eta';q)_{\infty}}\cdot O(\bk_{\uparrow})<+\infty.
\]
\end{proposition}
\begin{proof}
For positive integers $m,n$, we can check that the identity
\[
\sum_{n=1}^{\infty}\frac{[n]}{([n]-q^n\xi)([n]-q^n\eta)}\cdot c(m,n)=\frac{(1-q)q^m(q\xi';q)_m(q\eta';q)_m}{(q;q)_m(q\xi'\eta';q)_{m+1}}\cdot{}_2\phi_1\left({q\xi',q\eta' \atop q^{m+2}\xi'\eta'};q,q^m\right)
\]
holds by definition, where ${}_2\phi_1$ is the $q$-hypergeometric series.
By using Jacobi and Heine's $q$-Gauss summation formula for ${}_2\phi_1$, this value coincides with
\[
\frac{(1-q)q^m(q\xi';q)_m(q\eta';q)_m}{(q;q)_m(q\xi'\eta';q)_{m+1}}\cdot\frac{(q^{m+1}\xi';q)_{\infty}(q^{m+1}\eta';q)_{\infty}}{(q^{m+2}\xi'\eta';q)_{\infty}(q^m;q)_{\infty}}=\frac{q^m}{[m]}\cdot\frac{(q\xi';q)_{\infty}(q\eta';q)_{\infty}}{(q;q)_{\infty}(q\xi'\eta';q)_{\infty}}.
\]
Plugging this in for the definition of $Z(\bk;(1))$ yields the conclusion.
\end{proof}
As when we extend the domain of $O( \ )$, we identify a non-empty index $(k_1,\dots,k_r)$ and a word $yx^{k_1-1}\cdots yx^{k_r-1}\in y\bQ\langle x,y\rangle$ and understand $Z( \ ; \ )$ as a map from $(y\bQ\langle x,y\rangle)\times (y\bQ\langle x,y\rangle)$ to $\bR$ by the $\bQ$-bilinearity.
Notice that we don't substitute the empty index into $Z( \ ; \ )$ below.
Furthermore, for $w=\sum_{i=0}^{\infty}w_i\lambda^i\in (y\bQ\langle x,y\rangle)\jump{\lambda}$ and $w'=\sum_{j=0}^{\infty}w'_j\lambda^j\in(y\bQ\langle x,y\rangle)\jump{\lambda}$, we extend the definition of $Z( \ ; \ )$ by setting
\[
Z(w;w')\coloneqq\sum_{i=0}^{\infty}\sum_{j=0}^{\infty}Z(w_i,w'_j)\xi^{i+j}\eta^{i+j}
\]
as a possibly divergent series.
By this description, we rephrase the transport relations in Theorem \ref{thm:transport}.
Recall the definition of $\tau_{\lambda}$ and $A$ from the introduction.
We may assume $\varepsilon\coloneqq \xi\eta\cdot O((2))<1$ by the suitable choice of $\xi,\eta$.
\begin{lemma}\label{lem:rephrase}
Let $w, w'\in yA$ and $u\in A$.
\begin{enumerate}[(i)]
\item\label{item:1} $Z(w;w')$ is absolutely convergent.
\item\label{item:2} $Z(wu;w')=Z(w;w'\tau_{\lambda}(u))$.
\end{enumerate}
\end{lemma}
\begin{proof}
In order to prove \eqref{item:1}, it is enough to consider the case
\begin{align*}
w&=yu_1(1+yx\lambda)^{-1}u_2(1+yx\lambda)^{-1}\cdots u_n(1+yx\lambda)^{-1}u_{n+1},\\
w'&=yu'_1(1+yx\lambda)^{-1}u'_2(1+yx\lambda)^{-1}\cdots u'_m(1+yx\lambda)^{-1}u'_{m+1},
\end{align*}
where $u_1,\dots, u_n, u'_1,\dots, u'_m$ are words of $\bQ\langle x, y\rangle$.
In this case, we have, by definition,
\[
Z(w;w')=\sum_{i_1,\dots,i_{n+m}\geq 0}Z(w_{i_1,\dots,i_n};w'_{i_{n+1},\dots,i_{n+m}})\cdot(-\xi\eta)^{i_1+\cdots+i_{n+m}},
\]
where
\begin{align*}
&w_{i_1,\dots,i_n}=yu_1(yx)^{i_1}u_2(yx)^{i_2}\cdots u_n(yx)^{i_n}u_{n+1},\\
&w'_{i_{n+1},\dots,i_{n+m}}=yu'_1(yx)^{i_{n+1}}u'_2(yx)^{i_{n+2}}\cdots u'_m(yx)^{i_{n+m}}u_{m+1}.
\end{align*}
Since
\[
(0\leq) \ Z(w_{i_1,\dots,i_n};w'_{i_{n+1},\dots,i_{n+m}})\leq Z(yu_1u_2\cdots u_n; yu'_1u'_2\cdots u'_m)\cdot O((2))^{i_1+\cdots +i_{n+m}}
\]
holds, we have
\begin{align*}
&\sum_{i_1,\dots,i_{n+m}\geq 0}Z(w_{i_1,\dots,i_n};w'_{i_{n+1},\dots,i_{n+m}})\cdot(\xi\eta)^{i_1+\cdots+i_{n+m}}
\\
&\leq \sum_{i_1,\dots,i_{n+m}\geq 0}Z(yu_1u_2\cdots u_n; yu'_1u'_2\cdots u'_m)\cdot \varepsilon^{i_1+\cdots+i_{n+m}}\\
&=Z(yu_1u_2\cdots u_n; yu'_1u'_2\cdots u'_m)\cdot \left(\frac{1}{1-\varepsilon}\right)^{n+m},
\end{align*}
which proves \eqref{item:1}.

In order to prove \eqref{item:2}, thanks to \eqref{item:1}, we may assume without loss of generality that $w$ and $w'$ are words of $y\bQ\langle x,y\rangle$ and $u$ is either $x$ or $y$.
Under this assumption, we have, by the transport relations,
\[
Z(wy;w')\stackrel{\eqref{eq:TR1}}{=}Z(w;w'x)+\xi\eta\cdot Z(wyx;w'x) \stackrel{\eqref{eq:TR3}}{=}Z(w;w'(x(1+yx\lambda)))=Z(w;w'\tau_{\lambda}(y))
\]
and
\begin{align*}
Z(wx;w')&\stackrel{\eqref{eq:TR2}}{=}Z(w;w'y)-\xi\eta\cdot Z(wx;w'yx)\\
&=\cdots\\
&=\sum_{i=0}^{N}Z(w;w'(yx)^iy)\cdot(-\xi\eta)^i+Z(wx;w'(yx)^{N+1})\cdot(-\xi\eta)^{N+1}
\end{align*}
for every non-negative integer $N$.
Since
\[
Z(w;w'(yx)^iy)\leq Z(w;w'y)\cdot O((2))^i, \quad Z(wx;w'(yx)^{N+1})\leq Z(wx;w')\cdot O((2))^{N+1}
\]
and $\varepsilon=\xi\eta\cdot O((2))<1$, we obtain
\[
Z(wx;w')=\sum_{i=0}^{\infty}Z(w;w'(yx)^iy)\cdot(-\xi\eta)^i=Z(w;w'(1+yx\lambda)^{-1}y)=Z(w;w'\tau_{\lambda}(x))
\]
by letting $N$ tend to infinity.
\end{proof}
Now, we prove our main theorem.
Set $c_{\xi,\eta}\coloneqq\frac{(q;q)_{\infty}(q\xi'\eta';q)_{\infty}}{(q\xi';q)_{\infty}(q\eta';q)_{\infty}}$.
\begin{proof}[Proof of Theorem~$\ref{thm:main}$]
Let $w$ be an element of $A$.
By applying Proposition~\ref{prop:initial} and Lemma \ref{lem:rephrase}, we have the following equalities for absolutely convergent series:
\[
O(ywx)=c_{\xi,\eta}Z(yw;y)=c_{\xi,\eta}Z(y;y\tau_{\lambda}(w))=O(y\tau_{\lambda}(w)x).\qedhere
\]
\end{proof}
As already mentioned in the introduction, the double Ohno relation  (DOR) is a corollary of the extended DOR.
However we exhibit a direct proof of the original DOR below by transporting two arrows each time using Corollary~\ref{cor:transport}.
\begin{proof}[Proof of Theorem~$\ref{thm:genDOR}$]
Let $\bk$ be an index of BBBL-type:
\[
\bk=(\{2\}^{n_0},1,\{2\}^{n_1},3,\dots,\{2\}^{n_{2d-2}},1,\{2\}^{n_{2d-1}},3,\{2\}^{n_{2d}}).
\]
Then its dual is
\[
\bk^{\dagger}=(\{2\}^{n_{2d}},1,\{2\}^{n_{2d-1}},3,\dots,\{2\}^{n_{2}},1,\{2\}^{n_{1}},3,\{2\}^{n_{0}}).
\]
By Proposition~\ref{prop:initial}, the arrow notation, and the transport relations for the DOR, we can transport indices ``dynamically'' as follows:
\begin{align*}
O(\bk)&=O((1)_{\{\uparrow\to\}^{n_0}\{\to\uparrow\}^{n_1+1}\cdots \; \{\to\uparrow\}^{n_{2d-1}+1}\{\uparrow\to\}^{n_{2d}}\uparrow})\\
&=c_{\xi,\eta}\cdot Z((1)_{\{\uparrow\to\}^{n_0}\{\to\uparrow\}^{n_1+1}\cdots \; \{\to\uparrow\}^{n_{2d-1}+1}\{\uparrow\to\}^{n_{2d}}};(1))\\
&=c_{\xi,\eta}\cdot Z((1)_{\{\uparrow\to\}^{n_0}\{\to\uparrow\}^{n_1+1}\cdots \; \{\to\uparrow\}^{n_{2d-1}+1}\{\uparrow\to\}^{n_{2d}-1}};(1)_{\uparrow\to})\\
&=c_{\xi,\eta}\cdot Z((1)_{\{\uparrow\to\}^{n_0}\{\to\uparrow\}^{n_1+1}\cdots \; \{\to\uparrow\}^{n_{2d-1}+1}\{\uparrow\to\}^{n_{2d}-2}};(1)_{\uparrow\to\uparrow\to})\\
&=\cdots\\
&=c_{\xi,\eta}\cdot Z((1);(1)_{\{\uparrow\to\}^{n_{2d}}\{\to\uparrow\}^{n_{2d-1}+1}\cdots \; \{\to\uparrow\}^{n_{1}+1}\{\uparrow\to\}^{n_{0}}})\\
&=O((1)_{\{\uparrow\to\}^{n_{2d}}\{\to\uparrow\}^{n_{2d-1}+1}\cdots \; \{\to\uparrow\}^{n_{1}+1}\{\uparrow\to\}^{n_{0}}\uparrow})\\
&=O(\bk^{\dagger}),
\end{align*}
where $\{\uparrow\to\}^a$ means $\overbrace{\uparrow\to\cdots\uparrow\to}^{2a \ \text{arrows}}$ and so does $\{\to\uparrow\}^a$.
\end{proof}
\section{Another choice of parameters}
Our description of the definition of the connected sum/connector is a direct generalization of the one introduced by the third author and Yamamoto.
However, there appears a little complicated term $[m+n;\xi+\eta+(1-q)\xi\eta]$ in the definition of the connector $c(m,n)$.
In this section, we note another natural choice of parameters by which some formulas become slightly simplified.
Let $X, Y$ be non-negative real numbers satisfying
\[
\xi=\frac{q^{-X}-1}{1-q},\quad \eta=\frac{q^{-Y}-1}{1-q}.
\]
Note that $\xi'=q^{-X}$ and $\eta'=q^{-Y}$ by using the symbols in the previous section.
For a real number $a$ and a non-negative integer $n$, we set $\langle a\rangle\coloneqq 1-q^a$, $\langle a\rangle_n\coloneqq (q^{a};q)_n$.
Then, the generating function of the double Ohno sums and our connector are expressed as follows:
\[
O(\bk)=(1-q)^{k_1+\cdots+k_r}\sum_{0<m_1<\cdots<m_r}\prod_{i=1}^r\frac{q^{(k_i-1)m_i}}{\langle m_i-X \rangle\langle m_i-Y\rangle\langle m_i\rangle^{k_i-2}}
\]
and
\[
c(m,n)=\frac{q^{mn}\langle1-X\rangle_m\langle1-Y\rangle_m\langle1-X\rangle_n\langle1-Y\rangle_n}{\langle1\rangle_m\langle1\rangle_n\langle1-X-Y\rangle_{m+n}},
\]
where $\bk=(k_1,\dots,k_r)$ is an admissible index and $m, n$ are non-negative integers.

\end{document}